\documentclass[11pt]{article}

\usepackage{amsmath,amsthm}
\usepackage{amssymb}
\usepackage{hyperref}

\allowdisplaybreaks

\begin{document} 

\title{Convolution identities of poly-Cauchy numbers with level $2$
}  
\author{
Takao Komatsu\\
\small Department of Mathematical Sciences, School of Science\\[-0.8ex]
\small Zhejiang Sci-Tech University\\[-0.8ex]
\small Hangzhou 310018 China\\[-0.8ex]
\small \texttt{komatsu@zstu.edu.cn}
}

\date{
\small MR Subject Classifications: Primary 11B75; Secondary 11B37, 11B50, 05A15, 05A19, 11A55, 11C20 
}

\maketitle

\def\fl#1{\left\lfloor#1\right\rfloor}
\def\cl#1{\left\lceil#1\right\rceil}
\def\ang#1{\left\langle#1\right\rangle}
\def\stf#1#2{\left[#1\atop#2\right]} 
\def\sts#1#2{\left\{#1\atop#2\right\}}
\def\stff#1#2{\left[\!\!\left[#1\atop#2\right]\!\!\right]} 
\def\angg#1#2{\left\langle#1\atop#2\right\rangle}

\newtheorem{theorem}{Theorem}
\newtheorem{Prop}{Proposition}
\newtheorem{Cor}{Corollary}
\newtheorem{Lem}{Lemma}

\begin{abstract} 
Poly-Cauchy numbers with level $2$ are defined by inverse sine hyperbolic functions with the inverse relation from sine hyperbolic functions.  
In this paper, we show several convolution identities of poly-Cauchy numbers with level $2$. In particular, that of three poly-Cauchy numbers with level $2$ can be expressed as a simple form. In the sequel, we introduce the Stirling numbers of the first kind with level $2$.  
\end{abstract}


\section{Introduction}

Poly-Cauchy numbers (of the first kind) $c_n^{(k)}$ are defined as 
\begin{equation}
{\rm Lif}_k\bigl(\log(1+t)\bigr)=\sum_{n=0}^\infty c_n^{(k)}\frac{t^n}{n!}\,, 
\label{def:pcau} 
\end{equation}    
where ${\rm Lif}_k(z)$ is the {\it polylogarithm factorial} or {\it polyfactorial} function, defined by  
$$
{\rm Lif}_k(z)=\sum_{m=0}^\infty\frac{z^{m}}{m!(m+1)^k} 
$$ 
(\cite{Ko1,Ko2}).  
The concept of poly-Cauchy numbers with the polylogarithm factorial function is an anagloues of poly-Bernoulli numbers with the polylogarithm function (\cite{Kaneko}).  

There are many papers on poly-xxx numbers and most of them are just generalizations for generalization's sake, but this paper does not add another example.  For, most generalizations or variations of so-called poly numbers or polynomials are just with level $1$, but we consider poly numbers with level $2$.     

{\it Poly-Cauchy numbers}  $\mathfrak C_n^{(k)}$ {\it with level $2$} \cite{KP} are defined by  
\begin{equation}  
{\rm Lif}_{2,k}({\rm arcsinh}t)=\sum_{n=0}^\infty\mathfrak C_n^{(k)}\frac{t^n}{n!}\,,  
\label{def:pcaulevel2}  
\end{equation}  
where  ${\rm arcsinh}t$ is the inverse hyperbolic sine function and  
$$
{\rm Lif}_{2,k}(z)=\sum_{m=0}^\infty\frac{z^{2 m}}{(2 m)!(2 m+1)^k}\,. 
$$ 
which may be called the {\it polylogarithm factorial function with level $2$}.   
Notice that poly-Cauchy numbers with level $2$ are not simple generalizations of poly-Cauchy numbers or the original Cauchy numbers $c_n:=c_n^{(1)}$ defined by 
$$
\frac{t}{\log(1+t)}=\sum_{n=0}^\infty c_n\frac{t^n}{n!} 
$$  
because poly-Cauchy numbers with level $2$ are based on hyperbolic functions but so-called poly-Cauchy (or the original Cauchy) numbers with level $1$ are based upon logarithm functions.  In this sense, most generalizations of poly-Cauchy numbers are still based upon the same logarithm functions, but our poly-Cauchy numbers are constructed by a different function.    
In a similar sense, there are many generalizations of poly-Bernoulli numbers as with level $1$, but poly-cosecant numbers \cite{KPT} are those with level $2$.

The original Cauchy numbers, poly-Cauchy numbers and most of their generalizations are realted with the Stirling numbers of the first kind.  On the contrary, poly-Cauchy numbers with level $2$ are related with poly-Cauchy numbers with level $2$, which are not simple generalizations but essentially different from the original Stirling numbers of the first kind.  

In fact, $\mathfrak C_{n}^{(k)}$ has an expression in terms of $(2 m+1)^k$ ($m=1,2,\dots,n$) by using the Stirling numbers of the first kind with level $2$.  Note that $\mathfrak C_n=0$ for odd $n$.  

\begin{theorem}  
For integers $n$ and $k$ with $n\ge 1$,  
$$
\mathfrak C_{2 n}^{(k)}=\sum_{m=1}^n\frac{(-4)^{n-m}}{(2 m+1)^k}\stff{n}{m}\,,
$$ 
where for $m=1,2,\dots,n$
\begin{multline}
\stff{n}{m}=\stf{n}{m}^2-2\stf{n}{m-1}\stf{n}{m+1}+2\stf{n}{m-2}\stf{n}{m+2}\\
-\cdots+2(-1)^{m-1}\stf{n}{1}\stf{n}{2 m-1} 
\label{a:express}
\end{multline} 
with $\stff{n}{0}=0$. 
\label{th:form3} 
\end{theorem}

\noindent 
{\it Remark.}  
Poly-Cauchy numbers can be expressed by using the Stirling numbers of the first kind: 
$$
c_n^{(k)}=\sum_{m=0}^n\frac{(-1)^{n-m}}{(m+1)^k}\stf{n}{m}
$$ 
(\cite{Ko1}), where $\stf{n}{m}$ are the (unsigned) Stirling numbers of the first kind arise as coefficient of the rising factorial 
$$  
x(x+1)(x+2)\cdots(x+n-1)=\sum_{m=0}^n\stf{n}{m}x^m\,. 
$$ 

The Stirling numbers of the first kind $\stff{n}{m}$ with level $2$ arise as coefficient of the rising factorial 
\begin{equation}  
x(x+1^2)(x+2^2)\cdots(x+(n-1)^2)=\sum_{m=0}^n\stff{n}{m}x^m
\label{stff:risfac} 
\end{equation}
(see, e.g., \cite[p.213--217]{Riordan},\cite{BSSV}\footnote{There is a relation $\stff{n}{m}=(-1)^{n-m}t(2 n,2 m)$, where $t(n,m)$ are the central factorial numbers of the first kind, defined by $x(x+\frac{n}{2}-1)(x+\frac{n}{2}-2)\cdots(x-\frac{n}{2}+1)=\sum_{m=0}^n t(n,m)x^m$.}), 
So, they can also be written as 
$$
\stff{n}{m}=\sum_{1\le i_1<\cdots<i_{n-m}\le n-1}(i_1\cdots i_{n-m})^2\,.
$$ 
Thus, the following relation holds:  
\begin{equation}  
\stff{n}{m}=\stff{n-1}{m-1}+(n-1)^2\stff{n-1}{m}
\label{a:recrel}  
\end{equation} 
({\it Cf.}\cite{GZ}\footnote{In \cite{GZ} $u(n,m)$ are used as $u(n,m)=(-1)^{n-m}\stff{n}{m}$.}).  
In this sense, the numbers $\stff{n}{m}$ are suitable to be called the {\it Stirling numbers of the first kind with level $2$}, because the (unsigned) Stirling numbers of the first kind satisfy the recurrence relation
$$
\stf{n}{m}=\stf{n-1}{m-1}+(n-1)\stf{n-1}{m}\,. 
$$

Notice that concerning the Stirling numbers of the first kind we see 
\begin{align*}  
&\stf{n}{0}=0\quad(n\ge 1),\quad \stf{n}{1}=(n-1)!,\quad \stf{n}{2}=(n-1)!H_{n-1},\\ 
&\stf{n}{n}=1,\quad \stf{n}{n-1}=\binom{n}{2},\quad \stf{n}{n-2}=\frac{3 n-1}{4}\binom{n}{3},\quad \stf{n}{n-3}=\binom{n}{2}\binom{n}{4},\\
&\stf{n}{n-4}=\frac{15 n^3-30 n^2+5 n-2}{48}\binom{n}{5},\quad \stf{n}{n-5}=\frac{3 n^2-7 n-2}{8}\binom{n}{2}\binom{n}{6},\\ 
&\stf{n}{n-6}=\frac{63 n^5-315 n^4+315 n^3+91 n^2-42 n-16}{576}\binom{n}{7}\,. 
\end{align*} 
In particular, for $m=1,2,3$ 
\begin{align*}  
\stff{n}{1}&=\bigl((n-1)!\bigr)^2\,,\\
\stff{n}{2}&=\bigl((n-1)!\bigr)^2 H_{n-1}^{(2)}\quad{\rm \cite[A001819]{oeis}}\,,\\
\stff{n}{3}&=\bigl((n-1)!\bigr)^2\frac{(H_{n-1}^{(2)})^2-H_{n-1}^{(4)}}{2}\quad{\rm \cite[A001820]{oeis}}\,,\\
\end{align*}
where 
$$
H_n^{(k)}=\sum_{j=1}^n\frac{1}{j^k}
$$ 
is the generalized harmonic number of order $k$.  The numbers $\stff{n}{4}$ can be found in \cite[A001821]{oeis}. 
Since $\stf{n}{k}=0$ for $k>n>0$,  
\begin{align*}
\stff{n}{n}&=\stf{n}{n}^2=1\,,\\
\stff{n}{n-1}&=\stf{n}{n-1}^2-2\stf{n}{n-2}\stf{n}{n}=\frac{1}{2^2}\binom{2 n}{3}\,,\\
\stff{n}{n-2}&=\stf{n}{n-2}^2-2\stf{n}{n-3}\stf{n}{n-1}+2\stf{n}{n-4}\stf{n}{n}\\
&=\frac{5 n+1}{3\cdot 2^3}\binom{2 n}{5}\,,\\
\stff{n}{n-3}&=\stf{n}{n-3}^2-2\stf{n}{n-4}\stf{n}{n-2}+2\stf{n}{n-5}\stf{n}{n-1}-2\stf{n}{n-6}\stf{n}{n}\\
&=\frac{35 n^2+21 n+4}{9\cdot 2^4}\binom{2 n}{7}\,,\\
\stff{n}{n-4}&=\frac{(5 n+2)(35 n^2+28 n+9)}{15\cdot 2^5}\binom{2 n}{9}\,,\\
\stff{n}{n-5}&=\frac{385 n^4+770 n^3+671 n^2+286 n+48}{9\cdot 2^6}\binom{2 n}{11}\,. 
\end{align*} 

\begin{proof}[Proof of Theorem \ref{th:form3}.]  
First, notice that the expression of $\stff{n}{m}$ in terms of $\stf{n}{k}$ can satisfy the same recurrence relation in (\ref{a:recrel}).  
From (\ref{a:express}), since  
\begin{align*} 
&\stf{n-1}{m}\stf{n-1}{m-1}-\stf{n-1}{m-1}\stf{n-1}{m}+\stf{n-1}{m-2}\stf{n-1}{m+1}-\cdots\\
&\qquad -\stf{n-1}{m+1}\stf{n-1}{m-2}+\stf{n-1}{m+2}\stf{n-1}{m-3}-\cdots=0\,,
\end{align*}
we can also see that 
\begin{align*}  
\stff{n}{m}&=\left((n-1)\stf{n-1}{m}+\stf{n-1}{m-1}\right)^2\\
&\quad -2\left((n-1)\stf{n-1}{m-1}+\stf{n-1}{m-2}\right)\left((n-1)\stf{n-1}{m+1}+\stf{n-1}{m}\right)\\
&\quad +2\left((n-1)\stf{n-1}{m-2}+\stf{n-1}{m-3}\right)\left((n-1)\stf{n-1}{m+2}+\stf{n-1}{m+1}\right)-\cdots\\ 
&=\stff{n-1}{m-1}+(n-1)^2\stff{n-1}{m}\\
&\quad +2(n-1)\left(\stf{n-1}{m}\stf{n-1}{m-1}-\stf{n-1}{m-1}\stf{n-1}{m}+\stf{n-1}{m-2}\stf{n-1}{m+1}-\cdots\right.\\
&\qquad\left.-\stf{n-1}{m+1}\stf{n-1}{m-2}+\stf{n-1}{m+2}\stf{n-1}{m-3}-\cdots\right)\\
&=\stff{n-1}{m-1}+(n-1)^2\stff{n-1}{m}\,. 
\end{align*} 
  
Now, because 
$$
\frac{({\rm arcsinh} t)^{2 m}}{(2 m)!}=\sum_{n=m}^\infty(-4)^{n-m}\stff{n}{m}\frac{t^{2 n}}{(2 n)!}
$$ 
(see, e.g.\cite[(4.1.4)]{BSSV}\footnote{This proof is based upon the inverse relation between $\sin$ and arcsin with the orthogonal property of the central factorial numbers of both kinds.}), 
we have 
\begin{align*}  
\sum_{n=0}^\infty\mathfrak C_{2 n}\frac{t^{2 n}}{(2 n)!}&=\sum_{m=0}^\infty\frac{({\rm arcsinh} t)^{2 m}}{(2 m)!(2 m+1)^k}\\
&=\sum_{m=0}^\infty\frac{1}{(2 m+1)^k}\sum_{n=m}^\infty(-4)^{n-m}\stff{n}{m}\frac{t^{2 n}}{(2 n)!}\\
&=\sum_{n=0}^\infty\sum_{m=0}^n\frac{(-4)^{n-m}}{(2 m+1)^k}\stff{n}{m}\frac{t^{2 n}}{(2 n)!}\,.
\end{align*} 
Comparing the coefficients on both sides, we get the result.  
\end{proof} 
\bigskip

Poly-Cauchy numbers have an expression of integrals 
$$
c_n^{(k)}=n!\underbrace{\int_0^1\cdots\int_0^1}_k\binom{x_1 x_2\cdots x_k}{n}d x_1 d x_2\dots d x_k 
$$ 
(\cite{Ko1}).  
Poly-Cauchy numbers with level $2$ also have a similar expression (or a kind of definition). 

\begin{Cor}  
For $n\ge 0$ and $k\ge 1$, we have 
$$
\mathfrak C_{2 n}^{(k)}=(-4)^n(n!)^2\underbrace{\int_0^1\cdots\int_0^1}_k\binom{\dfrac{x_1 x_2\cdots x_k}{2}}{n}\binom{-\dfrac{x_1 x_2\cdots x_k}{2}}{n}d x_1 d x_2\dots d x_k\,. 
$$ 
\label{cor:form3}  
\end{Cor}  
\begin{proof}  
By Theorem \ref{th:form3} and the expression in (\ref{stff:risfac}), 
\begin{align*}  
&(-4)^n(n!)^2\underbrace{\int_0^1\cdots\int_0^1}_k\binom{\dfrac{x_1 x_2\cdots x_k}{2}}{n}\binom{-\dfrac{x_1 x_2\cdots x_k}{2}}{n}d x_1 d x_2\dots d x_k\\
&=\underbrace{\int_0^1\cdots\int_0^1}_k\sum_{m=0}^n(-4)^{n-m}\stff{n}{m}(x_1 x_2\cdots x_k)^{2 m}d x_1 d x_2\dots d x_k\\
&=\sum_{m=1}^n\frac{(-4)^{n-m}}{(2 m+1)^k}\stff{n}{m}=\mathfrak C_{2 n}^{(k)}\,. 
\end{align*}
\end{proof}



\section{Convolution}

When $k=1$, several initial values of $\mathfrak C_n=\mathfrak C_n^{(1)}$ are as follows.  
$$
\{\mathfrak C_{2 n}\}_{n\ge 0}=1,\frac{1}{3}, -\frac{17}{15}, \frac{367}{21}, -\frac{27859}{45}, \frac{1295803}{33}, -\frac{5329242827}{1365}, \dots\,. 
$$ 

In \cite{Zhao}, the convolution identity for Cauchy numbers are given as 
$$
\sum_{k=0}^n\binom{n}{k}c_k c_{n-k}=-n(n-2)c_{n-1}-(n-1)c_n\quad(n\ge 0)\,. 
$$ 
A more general case $\sum_{k=0}^n\binom{n}{k}c_{k+l}c_{n-k+m}$ for some fixed nonnegative integers $l$ and $m$ is treated in \cite{Ko5}.  In \cite{Ko6}, the convolution identities for Cauchy numbers of the second kind $\hat c_n$, defined by 
$$
\frac{t}{(1+t)\log(1+t)}=\sum_{n=0}^\infty\hat c_n\frac{t^n}{n!}
$$ 
have been studied.  In this section, we give the convolution identity for Cauchy numbers with level $2$ is given.   For simplicity, we use the conventional convolution notation 
$$ 
(\mathfrak C_{2 j_1}+\cdots+\mathfrak C_{2 j_k})^n:=\sum_{i_1+\cdots+i_k=n\atop i_1,\dots,i_k\ge 0}\binom{2 n}{2 i_1,\dots,2 i_k}\mathfrak C_{2 i_1+2 j_1}\cdots\mathfrak C_{2 i_k+2 j_k}\,,
$$ 
where 
$$
\binom{2 n}{2 i_1,\dots,2 i_k}=\frac{(2 n)!}{(2 i_1)!\cdots(2 i_k)!}
$$ 
is the multinomial coefficient.

\begin{theorem}  
For $n\ge 0$ 
$$
(\mathfrak C_{0}+\mathfrak C_{0})^n=(2 n)!\sum_{l=0}^n\frac{(-1)^{n-l}(2 n-2 l-3)!!(2 l-1)}{2^{n-l}(n-l)!(2 l)!}\mathfrak C_{2 l}\,. 
$$ 
Here $(2 i-1)!!=(2 i-1)(2 i-3)\cdots 1$ ($i\ge 1$) with $\bigl(-(2 i+1)\bigr)!!=\frac{(-1)^i}{(2 i-1)!!}$ ($i\ge 1$) and $(-1)!!=1$. 
\label{convo}  
\end{theorem} 
\begin{proof}  
For simplicity, put 
$$
L(t):=\frac{t}{{\rm arcsinh}t}=\sum_{n=0}^\infty\mathfrak C_{2 n}\frac{t^{2 n}}{(2 n)!}\,. 
$$ 
Since 
\begin{align*}  
L'(t)&=\frac{1}{{\rm arcsinh}t}-\frac{t}{\sqrt{1+t^2}({\rm arcsinh}t)^2}\\
&=\frac{1}{t}L(t)-\frac{1}{t\sqrt{1+t^2}}L(t)^2\,, 
\end{align*} 
we have 
\begin{equation} 
L(t)^2=-t\sqrt{1+t^2}L(t)'+\sqrt{1+t^2}L(t)\,. 
\label{eq:ll} 
\end{equation} 
Because 
\begin{align*}  
\sqrt{1+t^2}&=\sum_{j=0}^\infty\binom{\frac{1}{2}}{j}(t^2)^j\\
&=\sum_{j=0}^\infty\frac{(-1)^{j-1}(2 j-3)!!}{2^j\cdot j!}t^{2 j}
\end{align*}
and 
$$
tL(t)'=\sum_{l=0}^\infty(2 n)\mathfrak C_{2 n}\frac{t^{2 n}}{(2 n)!}\,,  
$$ 
we have 
\begin{align*}  
t\sqrt{1+t^2}L(t)'&=\left(\sum_{j=0}^\infty\frac{(-1)^{j-1}(2 j-3)!!}{2^j\cdot j!}t^{2 j}\right)\left(\sum_{l=0}^\infty(2 n)\mathfrak C_{2 n}\frac{t^{2 n}}{(2 n)!}\right)\\ 
&=\sum_{n=0}^\infty\left(\sum_{l=0}^n\frac{(-1)^{n-l-1}(2 n-2 l-3)!!(2 l)}{2^{n-l}(n-l)!}\frac{\mathfrak C_{2 l}}{(2 l)!}\right)t^{2 n} 
\end{align*}
and 
\begin{align*}  
t\sqrt{1+t^2}L(t)&=\left(\sum_{j=0}^\infty\frac{(-1)^{j-1}(2 j-3)!!}{2^j\cdot j!}t^{2 j}\right)\left(\sum_{l=0}^\infty\mathfrak C_{2 n}\frac{t^{2 n}}{(2 n)!}\right)\\ 
&=\sum_{n=0}^\infty\left(\sum_{l=0}^n\frac{(-1)^{n-l-1}(2 n-2 l-3)!!}{2^{n-l}(n-l)!}\frac{\mathfrak C_{2 l}}{(2 l)!}\right)t^{2 n} 
\end{align*}
Comparing the coefficients with 
$$
L(t)^2=\sum_{n=0}^\infty(\mathfrak C_{0}+\mathfrak C_{0})^n\frac{t^{2 n}}{(2 n)!}\,, 
$$ 
we get the result.  
\end{proof}

Since 
\begin{align}  
L(t)L''(t)&=\left(\frac{1}{2(1+t^2)^{3/2}}-\frac{1}{6\sqrt{1+t^2}}\right)L(t)\notag\\
&\quad +\left(\frac{\sqrt{1+t^2}}{6 t}+\frac{1}{2 t(1+t^2)^{3/2}}-\frac{2}{3 t\sqrt{1+t^2}}\right)L'(t)\notag\\
&\quad +\frac{1}{2}\left(\frac{1}{\sqrt{1+t^2}}-\sqrt{1+t^2}\right)L''(t)-\frac{t\sqrt{1+t^2}}{3}L^{(3)}(t)\,, 
\label{eq:convo02}
\end{align}
together with 
$$
\frac{1}{\sqrt{1+t^2}}=\sum_{j=0}^\infty\frac{(-1)^j(2 j-1)!!}{2^j\cdot j!}t^{2 j} 
$$ 
and 
$$
\frac{1}{(1+t^2)^{3/2}}=\sum_{j=0}^\infty\frac{(-1)^j(2 j+1)!!}{2^j\cdot j!}t^{2 j} 
$$ 
we have 
\begin{align*}  
&\left(\frac{1}{2(1+t^2)^{3/2}}-\frac{1}{6\sqrt{1+t^2}}\right)L(t)\\
&=\left(\sum_{j=0}^\infty\frac{(-1)^j(3 j+1)(2 j-1)!!}{3\cdot 2^j\cdot j!}t^{2 j}\right)\\
&\qquad\times\left(\sum_{l=0}^\infty\mathfrak C_{2 l}\frac{t^{2 l}}{(2 l)!}\right)\\
&=\sum_{n=0}^\infty\left(\sum_{l=0}^n\frac{(-1)^{n-l}(2 n-2 l-1)!!(3 n-3 l+1)}{3\cdot 2^{n-l}(n-l)!(2 l)!}\mathfrak C_{2 l}\right)t^{2 n}\,,   
\end{align*}  
\begin{align*}  
&\left(\frac{\sqrt{1+t^2}}{6 t}+\frac{1}{2 t(1+t^2)^{3/2}}-\frac{2}{3 t\sqrt{1+t^2}}\right)L'(t)\\
&=\left(\sum_{j=0}^\infty\frac{(-1)^j\bigl(-1+3(2 j+1)(2 j-1)-4(2 j-1)\bigr)(2 j-3)!!}{6\cdot 2^j\cdot j!}t^{2 j}\right)\left(\sum_{l=0}^\infty\mathfrak C_{2 l+2}\frac{t^{2 l}}{(2 l+1)!}\right)\\
&=\sum_{n=0}^\infty\left(\sum_{l=0}^n\frac{2(-1)^{n-l}(2 n-2 l-3)!!(n-l)(3 n-3 l-2)}{3\cdot 2^{n-l}(n-l)!(2 l+1)!}\mathfrak C_{2 l+2}\right)t^{2 n}\,,   
\end{align*}   
\begin{align*}  
&\frac{1}{2}\left(\frac{1}{\sqrt{1+t^2}}-\sqrt{1+t^2}\right)L''(t)\\
&=\left(\sum_{j=0}^\infty\frac{(-1)^j j(2 j-3)!!}{2^j\cdot j!}t^{2 j}\right)\left(\sum_{l=0}^\infty\mathfrak C_{2 l+2}\frac{t^{2 l}}{(2 l)!}\right)\\
&=\sum_{n=0}^\infty\left(\sum_{l=0}^n\frac{(-1)^{n-l}(2 n-2 l-3)!!(n-l)}{2^{n-l}(n-l)!(2 l)!}\mathfrak C_{2 l+2}\right)t^{2 n}    
\end{align*}
and 
\begin{align*}  
&\frac{t\sqrt{1+t^2}}{3}L^{(3)}(t)\\
&=\left(\sum_{j=0}^\infty\frac{(-1)^{j-1}(2 j-3)!!}{3\cdot 2^j\cdot j!}t^{2 j}\right)\left(\sum_{l=0}^\infty(2 l)\mathfrak C_{2 l+2}\frac{t^{2 l}}{(2 l)!}\right)\\
&=\sum_{n=0}^\infty\left(\sum_{l=0}^n\frac{(-1)^{n-l-1}(2 n-2 l-3)!!(n-l)(2 n)}{3\cdot 2^{n-l}(n-l)!(2 l)!}\mathfrak C_{2 l+2}\right)t^{2 n}\,.    
\end{align*} 
Thus, 
the right-hand side of (\ref{eq:convo02}) is equal to  
\begin{align*}  
&\sum_{n=0}^\infty\left(\sum_{l=0}^n\frac{(-1)^{n-l}(2 n-2 l-1)!!(3 n-3 l+1)}{3\cdot 2^{n-l}(n-l)!(2 l)!}\mathfrak C_{2 l}\right)t^{2 n}\\
&\quad +\sum_{n=0}^\infty\sum_{l=0}^n\frac{(-1)^{n-l}(6 n^2-6 n l+4 l^2-n+3 l)(2 n-2 l-3)!!}{3\cdot 2^{n-l}(n-l)!(2 l+1)!}\mathfrak C_{2 l+2}t^{2 n}\\
&=\sum_{n=0}^\infty\left(\sum_{l=0}^n\frac{(-1)^{n-l}(2 n-2 l-1)!!(3 n-3 l+1)}{3\cdot 2^{n-l}(n-l)!(2 l)!}\mathfrak C_{2 l}\right)t^{2 n}\\
&\quad +\sum_{n=0}^\infty\sum_{l=1}^{n+1}\frac{(-1)^{n-l-1}(6 n^2-6 n l+4 l^2+5 n-5 l+1)(2 n-2 l-1)!!}{3\cdot 2^{n-l+1}(n-l+1)!(2 l-1)!}\mathfrak C_{2 l}t^{2 n}\\
&=\sum_{n=0}^\infty\sum_{l=0}^{n+1}\frac{(-1)^{n-l-1}(2 l-1)(3 n^2-3 n l+2 l^2+4 n-3 l+1)(2 n-2 l-1)!!}{3\cdot 2^{n-l}(n-l+1)!(2 l)!}\mathfrak C_{2 l}t^{2 n}\,. 
\end{align*}
Since the left-hand side of (\ref{eq:convo02}) is equal to 
$$
\sum_{n=0}^\infty(\mathfrak C_{0}+\mathfrak C_{2})^n\frac{t^{2 n}}{(2 n)!}\,, 
$$ 
comparing the coefficients on both sides, we get the result of $(c_0+c_2)^n$.  

\begin{theorem}  
For $n\ge 0$,  
\begin{multline*}
(\mathfrak C_{0}+\mathfrak C_{2})^n\\ 
=(2 n)!\sum_{l=0}^{n+1}\frac{(-1)^{n-l-1}(2 l-1)(3 n^2-3 n l+2 l^2+4 n-3 l+1)(2 n-2 l-1)!!}{3\cdot 2^{n-l}(n-l+1)!(2 l)!}\mathfrak C_{2 l}\,. 
\end{multline*} 
\label{th:convo02}  
\end{theorem}

Similarly, the convolution of $(c_2+c_2)^n$ can be given as follows.  

\begin{theorem}  
For $n\ge 0$,  
\begin{align*}
&(\mathfrak C_{2}+\mathfrak C_{2})^n\\ 
&=\frac{(2 n)!}{30}\sum_{l=0}^{n}\frac{(-1)^{n-l}(10 n-8 l+5)(2 n-2 l-3)!!}{2^{n-l}(n-l)!(2 l)!}\mathfrak C_{2 l+4}\\
&\quad -\frac{(2 n)!}{3}\sum_{l=0}^{n}\frac{(-1)^{n-l}(6 l+1)(2 n-2 l+1)!!}{2^{n-l}(n-l)!(2 l)!}\mathfrak C_{2 l+2}\\
&\quad -\frac{(2 n)!}{30}\sum_{l=0}^{n}\frac{(-1)^{n-l}(160 l^3-220 l^2+72 l-1)(2 n-2 l+1)!!}{2^{n-l}(n-l)!(2 l)!}\mathfrak C_{2 l}\,. 
\end{align*} 
\label{th:convo22}  
\end{theorem} 

\begin{proof}  
We know that 
\begin{align*}
\bigl(L''(t)\bigr)^2&=-\frac{t\sqrt{1+t^2}}{30}L^{(5)}(t)+\frac{1}{6}\left(\frac{1}{\sqrt{1+t^2}}-2\sqrt{1+t^2}\right)L^{(4)}(t)\\
&\quad -\frac{3 t+2 t^2}{3(1+t^2)^{3/2}}L^{(3)}(t)-\frac{2+t^2}{6(1+t^2)^{3/2}}L''(t)-\frac{t}{30(1+t^2)^{3/2}}L'(t)\\
&\quad +\frac{1}{30(1+t^2)^{3/2}}L(t)\,.
\end{align*} 
Since 
\begin{align*}  
&-\frac{t\sqrt{1+t^2}}{30}L^{(5)}(t)\\
&=-\frac{1}{30}\sum_{j=0}^\infty\frac{(-1)^{j-1}(2 j-3)!!}{2^j\cdot j!}t^{2 j}\sum_{l=3}^\infty\mathfrak C_{2 l}\frac{t^{2 l-4}}{(2 l-5)!}\\
&=\frac{1}{30}\sum_{n=0}^\infty\left(\frac{(-1)^{n-l}(2 n-2 l-3)!!(2 l)}{2^{n-l}(n-l)!(2 l)!}\mathfrak C_{2 l+4}\right)t^{2 l}\,, 
\end{align*} 
\begin{align*}  
&\frac{1}{6}\left(\frac{1}{\sqrt{1+t^2}}-2\sqrt{1+t^2}\right)L^{(4)}(t)\\
&=-\frac{1}{6}\left(\sum_{j=0}^\infty\frac{(-1)^{j}(2 j-1)!!}{2^j\cdot j!}t^{2 j}\right.\\
&\qquad\qquad\left.-\sum_{j=0}^\infty\frac{(-1)^{j-1}2(2 j-3)!!}{2^j\cdot j!}t^{2 j}\right)\sum_{l=2}^\infty\mathfrak C_{2 l}\frac{t^{2 l-4}}{(2 l-4)!}\\
&=\frac{1}{6}\sum_{n=0}^\infty\left(\frac{(-1)^{n-l}(2 n-2 l-3)!!(2 n-2 l+1)}{2^{n-l}(n-l)!(2 l)!}\mathfrak C_{2 l+4}\right)t^{2 l}\,, 
\end{align*}
\begin{align*}  
&-\frac{3 t+2 t^2}{3(1+t^2)^{3/2}}L^{(3)}(t)\\
&=-\sum_{j=0}^\infty\frac{(-1)^{j}(2 j+1)!!}{2^j\cdot j!}t^{2 j}\sum_{l=2}^\infty\mathfrak C_{2 l}\frac{t^{2 l-2}}{(2 l-3)!}\\
&\quad -\frac{2}{3}\sum_{j=0}^\infty\frac{(-1)^{j}(2 j+1)!!}{2^j\cdot j!}t^{2 j}\sum_{l=2}^\infty\mathfrak C_{2 l}\frac{t^{2 l}}{(2 l-3)!}\\
&=-\sum_{n=0}^\infty\left(\frac{(-1)^{n-l}(2 n-2 l+1)!!}{2^{n-l}(n-l)!(2 l)!}\mathfrak C_{2 l+2}\right)t^{2 l}\\
&\quad -\frac{2}{3}\sum_{n=0}^\infty\left(\frac{(-1)^{n-l}(2 n-2 l+1)!!(2 l)(2 l-1)(2 l-2)}{2^{n-l}(n-l)!(2 l)!}\mathfrak C_{2 l}\right)t^{2 l}\,, 
\end{align*} 
\begin{align*}  
&--\frac{2+t^2}{6(1+t^2)^{3/2}}L''(t)\\
&=-\frac{2}{6}\sum_{j=0}^\infty\frac{(-1)^{j}(2 j+1)!!}{2^j\cdot j!}t^{2 j}\sum_{l=1}^\infty\mathfrak C_{2 l}\frac{t^{2 l-2}}{(2 l-2)!}\\
&\quad -\frac{1}{6}\sum_{j=0}^\infty\frac{(-1)^{j}(2 j+1)!!}{2^j\cdot j!}t^{2 j}\sum_{l=1}^\infty\mathfrak C_{2 l}\frac{t^{2 l}}{(2 l-2)!}\\ 
&=-\frac{1}{3}\sum_{n=0}^\infty\left(\frac{(-1)^{n-l}(2 n-2 l+1)!!}{2^{n-l}(n-l)!(2 l)!}\mathfrak C_{2 l+2}\right)t^{2 l}\\ 
&\quad -\frac{1}{6}\sum_{n=0}^\infty\left(\frac{(-1)^{n-l}(2 n-2 l+1)!!(2 l)(2 l-1)}{2^{n-l}(n-l)!(2 l)!}\mathfrak C_{2 l}\right)t^{2 l}  
\end{align*}
and  
\begin{align*}  
&-\frac{t}{30(1+t^2)^{3/2}}L'(t)+\frac{1}{30(1+t^2)^{3/2}}L(t)\\
&=-\frac{1}{30}\sum_{j=0}^\infty\frac{(-1)^{j}(2 j+1)!!}{2^j\cdot j!}t^{2 j}\sum_{l=1}^\infty\mathfrak C_{2 l}\frac{t^{2 l}}{(2 l-1)!}\\
&\quad +\frac{1}{30}\sum_{j=0}^\infty\frac{(-1)^{j}(2 j+1)!!}{2^j\cdot j!}t^{2 j}\sum_{l=0}^\infty\mathfrak C_{2 l}\frac{t^{2 l}}{(2 l)!}\\
&=-\frac{1}{30}\sum_{n=0}^\infty\left(\frac{(-1)^{n-l}(2 n-2 l+1)!!(2 l-1)}{2^{n-l}(n-l)!(2 l)!}\mathfrak C_{2 l}\right)t^{2 l}\,, 
\end{align*}
we have 
\begin{align*}  
&\bigl(L''(t)\bigr)^2\\
&=\frac{1}{30}\sum_{n=0}^\infty\left(\sum_{l=0}^{n}\frac{(-1)^{n-l}(10 n-8 l+5)(2 n-2 l-3)!!}{2^{n-l}(n-l)!(2 l)!}\mathfrak C_{2 l+4}\right)t^{2 l}\\
&\quad -\frac{1}{3}\sum_{n=0}^\infty\left(\sum_{l=0}^{n}\frac{(-1)^{n-l}(6 l+1)(2 n-2 l+1)!!}{2^{n-l}(n-l)!(2 l)!}\mathfrak C_{2 l+2}\right)t^{2 l}\\
&\quad -\frac{1}{30}\sum_{n=0}^\infty\left(\sum_{l=0}^{n}\frac{(-1)^{n-l}(160 l^3-220 l^2+72 l-1)(2 n-2 l+1)!!}{2^{n-l}(n-l)!(2 l)!}\mathfrak C_{2 l}\right)t^{2 l}\,. 
\end{align*} 
Since  
$$
\bigl(L''(t)\bigr)^2=\sum_{n=0}^\infty(\mathfrak C_{2}+\mathfrak C_{2})^n\frac{t^{2 n}}{(2 n)!}\,, 
$$ 
comparing the coefficients on both sides, we get the desired result.  
\end{proof}

\subsection{Higher-order convolutions}  

The convolution identity for three Cauchy numbers with level $2$ can be given as follows.  

\begin{theorem}  
For $n\ge 1$, 
\begin{align*}
&(\mathfrak C_{0}+\mathfrak C_{0}+\mathfrak C_{0})^n\\
&=(2 n-1)(n-1)\mathfrak C_{2 n}+n(2 n-1)(2 n-3)^2\mathfrak C_{2 n-2}\,. 
\end{align*} 
\label{th:convo000}
\end{theorem} 
\begin{proof}  
From (\ref{eq:ll}), 
$$
L(t)L(t)'=\frac{t}{2\sqrt{1+t^2}}L(t)-\frac{t^2}{2\sqrt{1+t^2}}L(t)'-\frac{t\sqrt{1+t^2}}{2}L(t)''\,. 
$$ 
So, 
\begin{align*}  
L(t)^3&=-t\sqrt{1+t^2}L(t)L(t)'+\sqrt{1+t^2}L(t)^2\\
&=-t\sqrt{1+t^2}\left(\frac{t}{2\sqrt{1+t^2}}L(t)-\frac{t^2}{2\sqrt{1+t^2}}L(t)'-\frac{t\sqrt{1+t^2}}{2}L(t)''\right)\\
&\quad +\sqrt{1+t^2}\bigl(-t\sqrt{1+t^2}L(t)'+\sqrt{1+t^2}L(t)\bigr)\\
&=\left(1+\frac{t^2}{2}\right)L(t)-\left(t+\frac{t^3}{2}\right)L(t)'+\frac{t^2(1+t^2)}{2}L(t)''\\
&=\sum_{n=0}^\infty\mathfrak C_{2 n}\frac{t^{2 n}}{(2 n)!}+\frac{1}{2}\sum_{n=0}^\infty(2 n+2)(2 n+1)\mathfrak C_{2 n}\frac{t^{2 n+2}}{(2 n+2)!}\\ 
&\quad -\sum_{n=1}^\infty\mathfrak C_{2 n}\frac{t^{2 n}}{(2 n-1)!}-\frac{1}{2}\sum_{n=0}^\infty(2 n+2)(2 n+1)(2 n)\mathfrak C_{2 n}\frac{t^{2 n+2}}{(2 n+2)!}\\ 
&\quad +\frac{1}{2}\sum_{n=1}^\infty\mathfrak C_{2 n}\frac{t^{2 n}}{(2 n-2)!}\\
&\quad +\frac{1}{2}\sum_{n=0}^\infty(2 n+2)(2 n+1)(2 n)(2 n-1)\mathfrak C_{2 n}\frac{t^{2 n+2}}{(2 n+2)!}\\ 
&=\sum_{n=0}^\infty\left(1-2 n+\frac{(2 n)(2 n-1)}{2}\right)\mathfrak C_{2 n}\frac{t^{2 n}}{(2 n)!}\\
&\quad +\frac{1}{2}\sum_{n=1}^\infty(2 n)(2 n-1)\bigl(1-(2 n-2)+(2 n-2)(2 n-3)\bigr)\mathfrak C_{2 n-2}\frac{t^{2 n}}{(2 n)!}\\
&=\sum_{n=0}^\infty(2 n-1)(n-1)\mathfrak C_{2 n}\frac{t^{2 n}}{(2 n)!}+\sum_{n=1}^\infty n(2 n-1)(2 n-3)^2\mathfrak C_{2 n-2}\frac{t^{2 n}}{(2 n)!}\,. 
\end{align*}
Comparing the coefficients with 
$$
L(t)^3=\sum_{n=0}^\infty(\mathfrak C_{0}+\mathfrak C_{0}+\mathfrak C_{0})^n\frac{t^{2 n}}{(2 n)!}\,, 
$$ 
we get the result.  
\end{proof}

The convolution identity for four Cauchy numbers with level $2$ can be given as follows.  

\begin{theorem}  
For $n\ge 1$, 
\begin{align*}
&(\mathfrak C_{0}+\mathfrak C_{0}+\mathfrak C_{0}+\mathfrak C_{0})^n\\
&=\frac{(2 n)!}{6}\sum_{l=0}^n\frac{(-1)^{n-l}(2 n-2 l-3)!!(2 l-1)(2 l-2)(2 l-3)}{2^{n-l}(n-l)!(2 l)!}\mathfrak C_{2 l}\\
&\quad +\frac{(2 n)!}{6}\sum_{l=1}^n\frac{(-1)^{n-l}(2 n-2 l-3)!!(2 l)(2 l-1)(2 l-3)^3}{2^{n-l}(n-l)!(2 l)!}\mathfrak C_{2 l-2}\,. 
\end{align*} 
\label{th:convo0000}
\end{theorem} 
\begin{proof}  
From the proof of Theorem \ref{th:convo000} and Theorem \ref{th:convo02}, we get 
\begin{align*}  
L(t)^4&=\left(1+\frac{t^2}{2}\right)L(t)-\left(t+\frac{t^3}{2}\right)L(t)'+\frac{t^2(1+t^2)}{2}L(t)''\\
&=\frac{(6+t^2)\sqrt{1+t^2}}{6}L(t)-\frac{t(6+t^2)\sqrt{1+t^2}}{6}L'(t)\\
&\quad +\frac{t^2\sqrt{1+t^2}}{2}L''(t)-\frac{(t^3+t^5)\sqrt{1+t^2}}{6}L^{(3)}\,.
\end{align*} 
Since 
\begin{align*}  
&\frac{(6+t^2)\sqrt{1+t^2}}{6}L(t)\\ 
&=\left(\sum_{j=0}^\infty\frac{(-1)^{j-1}(2 j-3)!!}{2^j\cdot j!}t^{2 j}\right)\left(\sum_{l=0}^\infty\mathfrak C_{2 l}\frac{t^{2 l}}{(2 l)!}\right)\\
&\quad +\frac{1}{6}\left(\sum_{j=0}^\infty\frac{(-1)^{j-1}(2 j-3)!!}{2^j\cdot j!}t^{2 j}\right)\left(\sum_{l=1}^\infty(2 l)(2 l-1)\mathfrak C_{2 l-2}\frac{t^{2 l}}{(2 l)!}\right)\,, 
\end{align*} 
\begin{align*}  
&-\frac{t(6+t^2)\sqrt{1+t^2}}{6}L'(t)\\ 
&=-\left(\sum_{j=0}^\infty\frac{(-1)^{j-1}(2 j-3)!!}{2^j\cdot j!}t^{2 j}\right)\left(\sum_{l=0}^\infty\mathfrak C_{2 l}\frac{t^{2 l}}{(2 l)!}\right)\\
&\quad -\frac{1}{6}\left(\sum_{j=0}^\infty\frac{(-1)^{j-1}(2 j-3)!!}{2^j\cdot j!}t^{2 j}\right)\left(\sum_{l=1}^\infty(2 l)(2 l-1)\mathfrak C_{2 l-2}\frac{t^{2 l}}{(2 l)!}\right)\,, 
\end{align*} 
\begin{align*}  
&\frac{t^2\sqrt{1+t^2}}{2}L''(t)\\ 
&=\frac{1}{2}\left(\sum_{j=0}^\infty\frac{(-1)^{j-1}(2 j-3)!!}{2^j\cdot j!}t^{2 j}\right)\left(\sum_{l=0}^\infty(2 l)(2 l-1)\mathfrak C_{2 l}\frac{t^{2 l}}{(2 l)!}\right) 
\end{align*} 
and 
\begin{align*}  
&-\frac{(t^3+t^5)\sqrt{1+t^2}}{6}L^{(3)}\\ 
&=-\frac{1}{6}\left(\sum_{j=0}^\infty\frac{(-1)^{j-1}(2 j-3)!!}{2^j\cdot j!}t^{2 j}\right)\left(\sum_{l=0}^\infty\frac{(2 l)!}{(2 l-3)!}\mathfrak C_{2 l}\frac{t^{2 l}}{(2 l)!}\right)\\
&\quad -\frac{1}{6}\left(\sum_{j=0}^\infty\frac{(-1)^{j-1}(2 j-3)!!}{2^j\cdot j!}t^{2 j}\right)\left(\sum_{l=1}^\infty\frac{(2 l)!}{(2 l-5)!}\mathfrak C_{2 l-2}\frac{t^{2 l}}{(2 l)!}\right)\,, 
\end{align*} 
we have 
\begin{align*}  
L(t)^4 
&=\frac{1}{6}\left(\sum_{j=0}^\infty\frac{(-1)^{j}(2 j-3)!!}{2^j\cdot j!}t^{2 j}\right)\left(\sum_{l=0}^\infty(2 l-1)(2 l-2)(2 l-3)\mathfrak C_{2 l}\frac{t^{2 l}}{(2 l)!}\right)\\
&\quad +\frac{1}{6}\left(\sum_{j=0}^\infty\frac{(-1)^{j}(2 j-3)!!}{2^j\cdot j!}t^{2 j}\right)\left(\sum_{l=1}^\infty(2 l)(2 l-1)(2 l-3)^3\mathfrak C_{2 l-2}\frac{t^{2 l}}{(2 l)!}\right)\\
&=\frac{1}{6}\sum_{n=0}^\infty\sum_{l=0}^n\frac{(-1)^{n-l}(2 n-2 l-3)!!(2 l-1)(2 l-2)(2 l-3)}{2^{n-l}(n-l)!(2 l)!}\mathfrak C_{2 l}t^{2 n}\\
&\quad +\frac{1}{6}\sum_{n=1}^\infty\sum_{l=1}^n\frac{(-1)^{n-l}(2 n-2 l-3)!!(2 l)(2 l-1)(2 l-3)^3}{2^{n-l}(n-l)!(2 l)!}\mathfrak C_{2 l-2}t^{2 n}\,,
\end{align*} 
Comparing the coefficients with 
$$
L(t)^4=\sum_{n=0}^\infty(\mathfrak C_{0}+\mathfrak C_{0}+\mathfrak C_{0}+\mathfrak C_{0})^n\frac{t^{2 n}}{(2 n)!}\,, 
$$ 
we get the result.  
\end{proof}

As seen, convolution identities of the odd number of Cauchy numbers are simpler than those of the even number.    
Similarly, from 
\begin{align*}  
L(x)^5&=\frac{x^4(1+x^2)^2}{4!}L^{(4)}(x)+\frac{x^3(x^2-2)(1+x^2)}{12}L^{(3)}(x)\\
&\quad +\frac{x^2(x^4+10 x^2+12)}{4!}L''(x)-\frac{x(x^4+20 x^2+24)}{4!}L'(x)\\
&\quad+\frac{x^4+20 x^2+24}{4!}L(x)\,, 
\end{align*} 
we have for $n\ge 2$
\begin{align*}  
&(\mathfrak C_{0}+\mathfrak C_{0}+\mathfrak C_{0}+\mathfrak C_{0}+\mathfrak C_{0})^n\\
&=\binom{2 n-1}{4}\mathfrak C_{2 n}  
+\frac{4 n^2-16 n+17}{3}\binom{2 n}{2}\binom{2 n-3}{2}\mathfrak C_{2 n-2}\\
&\quad +\binom{2 n}{4}(2 n-5)^4\mathfrak C_{2 n-4}\,. 
\end{align*} 
From 
\begin{align*}  
L(x)^7&=\frac{x^6(1+x^2)^3}{6!}L^{(6)}(x)+\frac{x^5(3 x^2-2)(1+x^2)^2}{2\cdot 5!}L^{(5)}(x)\\
&\quad +\frac{x^4(4 x^4+x^2+6)(1+x^2)}{3!4!}L^{(4)}(x)\\
&\quad +\frac{x^3(2 x^2+3)(x^4-4 x^2-8)}{4!3!}L^{(3)}(x)\\
&\quad +\frac{x^2(x^6+91 x^4+420 x^2+360)}{6!}L''(x)\\
&\quad +\frac{x(x^6+182 x^4+840 x^2+720)}{6!}L'(x)\\
&\quad +\frac{x^6+182 x^4+840 x^2+720}{6!}L(x)\,,  
\end{align*}
we have for $n\ge 3$ 
\begin{align*}  
&(\underbrace{\mathfrak C_{0}+\cdots+\mathfrak C_{0}}_7)^n\\
&=\binom{2 n-1}{6}\mathfrak C_{2 n}+\frac{12 n^2-60 n+83}{15}\binom{2 n}{2}\binom{2 n-3}{4}\mathfrak C_{2 n-2}\\
&\quad +\frac{(4 n^2-24 n+39)(12 n^2-72 n+109)}{15}\binom{2 n}{4}\binom{2 n-5}{2}\mathfrak C_{2 n-4}\\
&\quad +\binom{2 n}{6}(2 n-7)^6\mathfrak C_{2 n-6}\,. 
\end{align*} 

Nevertheless, the higher-order cases seem to be more complicated when the number of Cauchy numbers increases.  
It is expected that for any integer $r\ge 1$ 
$$
(\underbrace{\mathfrak C_{0}+\cdots+\mathfrak C_{0}}_{2 r+1})^n
=\sum_{k=0}^r P_{r, 2k}(n)\binom{2 n}{2 k}\binom{2 n-2 k-1}{2 r-2 k}\mathfrak C_{2 n-2k}\,,
$$ 
where $P_{r,2 k}(n)$ are the polynomials of $n$ with degree $2 k$ ($0\le k\le r$). In particular, $P_{r,0}(n)=1$ and $P_{r, 2 r}=(2 n-2 r-1)^{2 r}$.


\begin{thebibliography}{99} 

\bibitem{BMS} 
P. L. Butzer, C. Markett and M. Schmidt, {\em  
Stirling numbers, central factorial numbers, and representations of the Riemann zeta function}, 
Results Math. {\bf 19} (1991), 257--274.   

\bibitem{BSSV}  
P. L. Butzer, M. Schimidt, E. L. Stark and L. Vogt, {\em  
Central factorial numbers; their main properties and some applications},  
Numer. Funct. Anal. Optimiz. {\bf 10} (1989), 419--488.  

\bibitem{GZ}  
Y. Gelineau and J. Zeng, {\em  
Combinatorial interpretations of the Jacobi-Stirling numbers},  
Electron. J. Combin. {\bf 17} (2010), Paper \#R70.

\bibitem{Kaneko}  
M. Kaneko, {\em 
Poly-Bernoulli numbers}, 
J. Th\'eor. Nombres Bordeaux {\bf 9} (1997), 199--206.   

\bibitem{KPT}  
M. Kaneko, M. Pallenwatta and H. Tsumura, {\em  
On poly-cosecant numbers},  
arXiv:1907.13441 (2019).  

\bibitem{Ko1}  
T. Komatsu, {\em  
Poly-Cauchy numbers},  
Kyushu J. Math. {\bf 67} (2013), 143--153. 

\bibitem{Ko2}  
T. Komatsu, {\em  
Poly-Cauchy numbers with a $q$ parameter},   
Ramanujan J. {\bf 31} (2013), 353--371.  

\bibitem{Ko5}  
T. Komatsu, {\em  
Convolution identities for Cauchy numbers},  
Acta Math. Hungar. {\bf 144} (2014), 76--91. 

\bibitem{Ko6}  
T. Komatsu, {\em  
Convolution identities for Cauchy numbers of the second kind},  
Kyushu J. Math. {\bf 69} (2015), 125--144.  

\bibitem{KP}  
T. Komatsu and C. Pita-Ruiz,  {\em  
Poly-Cauchy numbers with level $2$},  
Integral Transforms Spec. Func. (published online).  
https://doi.org/10.1080/10652469.2019.1710745 

\bibitem{Riordan} 
J. Riordan, {\em 
Combinatorial Identities}, 
John Wiley \& Sons, Inc., 1968. 

\bibitem{oeis}  
N. J. A. Sloane, {\em  
The On-Line Encyclopedia of Integer Sequences},  
available at oeis.org.  (2020).  
 
\end{thebibliography}
\end{document}